\documentclass{amsart}
 \usepackage{amsmath,amsfonts,amssymb, amsthm, color, cite, mathrsfs}
\usepackage{epsfig}
\usepackage{verbatim}
\usepackage{booktabs}
\usepackage{multirow}

\usepackage[latin1]{inputenc}
\usepackage[T1]{fontenc}

\usepackage[english]{babel}

\usepackage{subfig}

\usepackage{graphicx} 

\usepackage{hyperref}

\usepackage{fancyhdr}

\newtheorem{thm}{\bf{Theorem}}[section]
\newtheorem{lemma}[thm]{Lemma}
\newtheorem{cor}[thm]{Corollary}
\newtheorem{prop}[thm]{Proposition}

\theoremstyle{remark}
\newtheorem{rem}[thm]{Remark}
\theoremstyle{definition}
\newtheorem{defin}[thm]{Definition}

\newtheorem{proof*}{Proof}
\newtheorem*{thm*}{Theorem}

\newcommand{\GL}{\operatorname{GL}}
\newcommand{\SL}{\operatorname{SL}}
\newcommand{\Pic}{\operatorname{Pic}}
\newcommand{\Bun}{\operatorname{Bun}}

\newcommand{\todo}[1]{\marginpar{{\color{red}\tiny{#1}}}}

\newcommand{\Hom}{\mathrm{Hom}}

\newcommand{\ov}{\overline}

\renewcommand{\ss}{\mathrm{ss}}

\providecommand{\gcd}{\mathrm{g.c.d.}}

\renewcommand{\L}{\mathcal{L}} \newcommand{\cL}{\mathcal{L}}

\newcommand{\cO}{\mathcal{O}}

\newcommand{\<}{\langle}
\renewcommand{\>}{\rangle}

\newcommand{\cF}{\mathcal{F}}

\newcommand{\Z}{\mathbb{Z}}
\newcommand{\Q}{\mathbb{Q}}

 \newcommand{\bT}{\mathbb{T}}

\newcommand{\bG}{\mathbb{G}}

\newcommand{\F}{\mathcal{F}}
\newcommand{\V}{\mathcal{V}}
\providecommand{\L}{\mathcal{L}}

\providecommand{\ss}{\mathsf{ss}}
\newcommand{\st}{\mathsf{st}}
\newcommand{\pos}{\mathsf{pos}}
\newcommand{\reg}{\mathsf{reg}}
\newcommand{\ad}{\mathsf{ad}}

\newcommand{\Lam}{\Lambda}

\newcommand{\PGL}{\operatorname{PGL}}
\newcommand{\proj}{\operatorname{proj}}

\newcommand{\g}{\mathfrak{g}}
\newcommand{\p}{\mathfrak{p}}

\newcommand{\ari}{\ar@{^{(}->}}
\newcommand{\aril}{\ar@{_{(}->}}
\newcommand{\cart}[1]{\ar@{}[#1]|\square }

\providecommand{\Sigma}{\mathfrak{S}}

\newcommand{\om}{\omega}

\newcommand{\I}{\mathcal{I}}

\newcommand{\oBun}{\overline{\Bun}}
\newcommand{\vLam}{\check{\Lambda}}

\newcommand{\valpha}{{\check{\alpha}}}

\newcommand{\vlam}{{\check{\lambda}}}
\newcommand{\lam}{\lambda}
\newcommand{\vmu}{{\check{\mu}}}

\title{On the stack of semistable $G$-bundles over an elliptic curve}
\author{Dragos Fratila}
\address{Max Planck Institut f\"ur Mathematik, 7 Vivatsgasse, Bonn 53111, Germany} 

 \email{dragos.fratila@gmail.com}

\numberwithin{equation}{section}

\input xy
\xyoption {all} 

\begin{document}

\begin{abstract}
In a recent paper Ben-Zvi and Nadler proved that the induction map from $B$-bundles of degree 0 to semistable $G$-bundles of degree 0 over an elliptic curve is a small map with Galois group isomorphic to the Weyl group of $G$. 
We generalize their result to all connected components of $\Bun_G$ for an arbitrary reductive group $G$. We prove that for every degree (i.e. topological type) there exists a unique parabolic subgroup such that any semistable $G$\nobreakdash-bundle of this degree has a reduction to it and moreover the induction map is small with Galois group the relative Weyl group of the Levi. This provides new examples of simple automorphic sheaves which are constituents of Eisenstein sheaves for the trivial local system. 
\end{abstract}

\maketitle

\normalsize

\setcounter{tocdepth}{1} 

\tableofcontents  

\section{Introduction}
The moduli spaces of semistable principal bundles on elliptic curves have received considerable attention in the past years. 
Starting with the result of Atiyah~\cite{Ati} who classified semistable $GL_n$ bundles and then continuing with the work of Tu~\cite{Tu_ell},  Laszlo~\cite{Lasz} and Friedman, Morgan, Witten \cite{FMW1, FMW2, FMW3} who classified and studied principal (semistable) bundles on elliptic curves and elliptic fibrations. 
The latter papers were mostly motivated by applications to physics, more precisely $F$-theory.
Although the result of Atiyah can be upgraded to a stacky statement using Fourier\nobreakdash-Mukai transforms, the other approaches for general reductive groups are for moduli spaces and not much attention has been given to moduli stacks.  
One shortcome of these approaches, unlike the case of $\GL_n$, is that one cannot apply these results in the study of geometric Eisenstein series for elliptic curves. 
Our initial motivation comes from the desire to use the geometry of the moduli stacks of semistable $G$-bundles in the classification of simple constituents of Eisenstein sheaves for an elliptic curve and for the trivial local system (see \cite{Sch2} for a treatment in the case of an elliptic curve and $\GL_n$ and \cite{Lau} for the projective line). 

\vspace{0.1in}
\vspace{0.1in}

More recently, Ben-Zvi and Nadler~\cite{BZN} started the study of the moduli stack of principal $G$-bundles of degree $0$ over an elliptic curve. 
They proved that the induction map from $B$-bundles of degree $0$ to semistable $G$-bundles of degree 0 is a small map with Galois group the Weyl group of $G$. 
Their motivation stems from the idea of constructing character sheaves for loop groups using principal bundles over a genus one curve (this was first suggested by V. Ginzburg).  
Indeed, some evidence for this comes from a result of Looijenga (unpublished, but see~\cite{BaGi_conj}) who proved that there is a bijection between the isomorphism classes of holomorphic principal $G$-bundles on an elliptic curve and (twisted) conjugacy classes in the holomorphic loop group. 
This is backed-up by the construction of character sheaves as the center of the Hecke category (see \cite{BeFiOs},\cite{BZN_cth}) and by the, currently developping, affine analog (see \cite{BZNP_affch}). 
Our motivation also comes, partly, from this perspective. 

The geometry studied in this note is one of the ingredients for the classification theorem of the simple constituents of spherical Eisenstein sheaves that we propose in~\cite{Frat_eis}. In particular, for every Harder-Narasimhan stratum of $\Bun_G$ we obtain simple automorphic sheaves supported on it (with monodromy given by representations of a relative Weyl group) and appearing as direct summands in the Eisenstein sheaves of the trivial local system.

\bigskip

Here is the statement of our main result (for notations see Section~\ref{sec_prelim}):

\begin{thm*}(Theorem~\ref{thm_main}) Let $X$ be an elliptic curve and $\vlam_G\in\vLam_{G,G}$ a degree (a cocharacter modulo the coroots). Then there exists a smallest parabolic subgroup $P$ and a unique degree $\vlam_P\in\vLam_{G,P}$ such that the induction map restricted to the semistable locus
\[
p:\Bun_P^{\vlam_P,\ss}\to \Bun_G^{\vlam_G,\ss} 
\]
is proper, small and generically Galois with Galois group the relative Weyl group $W_{M,G} = N_G(M)/M$ where $M$ is the Levi subgroup of $P$.
\end{thm*}
The parabolic $P$ and the degree $\vlam_P$ are unique (up to conjugation) and are given by Lemma~\ref{the_parabolic}.

In the case of degree $0$ the parabolic in question is the Borel subgroup and we recover the result of Ben-Zvi and Nadler~\cite{BZN}. 

For example, if the group is $G=\GL_6$ and the degree is $2$ then the parabolic is formed by the upper block-matrices $3\times 3$ and the degree is $(1,1)$. 
If the degree is $3$ then the parabolic is formed by the (upper) block-matrices $2\times 2$ and the degree is $(1,1,1)$. 
In the case of $\GL_n$ 
it is not difficult to find the parabolic based on the $\gcd(n,d)$ and our main theorem for $\GL_n$ follows easily from some dimension estimates that are spelled out in~\cite[Proposition 4.3.1]{Lau}. 

We refer to the table at the end of the paper where the full list of Levi subgroups is provided.

\begin{rem}
In the case of $\GL_n$ it makes sense to talk about Jordan-H\"older series in the category of semistable vector bundles of fixed slope. 
The above theorem provides also a generalization in the context of semistable $G$-bundles of this Jordan-H\"older series. 
This idea has already appeared in the pioneering work of Ramanathan, see \cite[Section 3]{Ram-modI}.
\end{rem}  
  
\vspace{0.1in}

The Levi subgroups that appear in our main theorem are exactly those Levi from the generalized Springer correspondence (see \cite{Lu_gen_spr}) for the Langlands dual group $\vphantom{G}^\mathsf{L}\! G$ that admit a cuspidal local system on the unipotent regular orbit. At the moment we are unable to understand the precise reasons for this mysterious coincidence.  
There is a simple combinatorial description of these Levi subgroups which is given in \cite{Bon} as well as a complete classification.
Our Levi subgroups are also defined combinatorially and it can be proved (elementary) that the combinatorial problem that defines them (Lemma~\ref{the_parabolic}) is equivalent to the combinatorial problem studied in~\cite{Bon}.
I'm grateful to C. Bonnaf\'e who explained this to me. 

\bigskip

Let us briefly outline the contents of this note. 
In Section 2 we recall some basic facts about the stacks of principal bundles over curves, we introduce the slope map $\phi_P$ (following~\cite{Schi}) and we prove the main combinatorial lemma (Lemma~\ref{the_parabolic}). 
In Section 3 we state our main result and prove it through a series of lemmas, some of which apply to curves of arbitrary genus and hence could be of independent interest. 
In Section 4 we (re)prove, as a simple application of our main result, that over an elliptic curve there are no stable $G$-bundles unless $G$ is of type $A$ (see Corollary~\ref{cor_stable}). We also provide a table including all possible Levi subgroups appearing in our main Theorem~\ref{thm_main}.


\vspace{0.1in}

\section{Preliminaries}\label{sec_prelim}
\subsection{Notations and conventions}
We will work over an algebraically closed field $k$ of characteristic 0. 
Throughout $X$ will be an arbitrary smooth projective curve, geometrically irreducible. 
We will emphasize the precise results where we need $X$ to be of genus 1. 

By $G$ we will denote an arbitrary reductive group over $k$. We will fix a maximal torus and a Borel subgroup $T\subset B\subset G$.

We will denote by $\Lam_G=\Hom(T,\bG_m)$ and $\vLam_G=\Hom(\bG_m,T)$ the lattices of characters, respectively cocharacters of $T$ (in \cite{BG1} the notations are interchanged). 

The roots of $G$, denoted here by $\Phi$, are the non-trivial characters of $T$ appearing in the representation of $T$ on $\g$. The choice of a Borel subgroup gives us a partition of $\Phi=\Phi^+\sqcup \Phi^-$ into positive and negative roots as well as a set of simple positive roots $\{\alpha_i:\, i\in \I\}\subset \Phi^+$ where $\I$ is the set of vertices of the Dynkin diagram. Similarly, we denote by $\{\valpha_i:\, i\in \I\}$ the set of simple coroots.

The parabolic subgroups of $G$ that contain $B$ are in bijection with the subsets of $\I$. A parabolic subgroup has a reductive Levi quotient $M:=P/R_u(P)$ where $R_u$ is the unipotent radical of $P$. For a parabolic subgroup $P$ we will denote by $\I_P\subset \I$ the set of simple roots of its Levi $M$. Conversely, any subset of $\I$ defines a parabolic subgroup of $G$ that contains the Borel subgroup $B$. 
These parabolics are called standard and it is well known that any parabolic subgroup is conjugated to one of the standard ones. 
The Levi quotient $M$ has a canonical splitting $M\to P$ coming from the fact that we fixed a maximal torus and a Borel subgroup (it is constructed using the $\SL_2$-triples associated to $i\in \I_P$). 
We will view $M$ either as a quotient of $P$ or as a subgroup of $P$ using this splitting.

\vspace{0.3cm}

We put $\vLam_{G,P} := \vLam_G\slash \mathrm{span}\{\valpha_i:i\in\I_P\}$ and $\Lam_{G,P} := \{\lambda\in\Lam_G: \<\lambda,\valpha_i\> = 0,\forall i\in\I_P\}$. 
Whenever we adorn the above $\Z$-modules by an upper index $\Q$ it means that we tensor them over $\Z$ with $\Q$. We fix fundamental weights $\omega_i\in \Lam_G^{\Q}$ such that $\<\omega_i,\vLam_G\>\in\Z$ (this is automatic for $G$ semisimple but when $G$ is reductive the characters of the center give more freedom). 

The semigroup of dominant characters is denoted by $\Lam_G^+=\{\lam\in\Lam_G: \<\lam,\valpha_i\>\ge 0\}$ and for positive characters we use $\Lam_G^\pos:=\Z_+\{\alpha_i\mid i\in I\}$. Similarly one defines $\vLam_G^\pos$. We denote by $\vLam_G^{\pos,\Q}=\Q_+ \cdot\vLam_G^{\pos}\subset \vLam_G^\Q$ the rational cone of positive coweights. 
The latter induces a partial order on $\vLam_G^\Q$: $\vlam\le \vmu$ if $\vmu-\vlam\in\vLam_G^{\pos,\Q}$. 
The Weyl module of $G$ of highest weight $\lam\in\Lam_G^+$ is denoted by $V^\lam$. 
Also, when $\lam$ is a character of one of the groups $T,P$ or $G$ we denote by $L^\lam$ the corresponding one\nobreakdash-dimensional representation.

\vspace{0.3cm}

For an affine group $H$, principal $H$-bundles over a scheme (or a stack) $S$ are to be understood as schemes (or stacks) $Y$ together with a right action of $H$ and an $H$-invariant map $p:Y\to S$ such that, locally (in the fppf topology on $S$), $p:Y\to S$ is equivariantly isomorphic to $p_2:S\times H\to S$.

We will denote by $BH$ the classifying stack of $H$-bundles. It is the functor that to a scheme $S$ associates the groupoid of $H$-bundles on $S$ where all the morphisms are isomorphisms of $H$-bundles.

We denote by $\Bun_H=\underline{Hom}(X,BH)$ the moduli stack of $H$-bundles on $X$: the $S$-points of $\Bun_H$ are $H$-bundles on $X\times S$ (note that it is automatically flat over $S$) together with isomorphisms of $H$-bundles.

The connected components of the stack $\Bun_G$ are in bijection with the set $\vLam_{G,G}$. Similarly, the connected components of the stack $\Bun_P$ are in bijection with $\vLam_{G,P}=\vLam_{M,M}$ where $M$ is the Levi of $P$. For a nice treatment of these results we refer to \cite{Hoff}. 

We will call an element of $\vLam_{G,P}$ a degree and refer to the degree of a principal $P$-bundle $\F_P$ as the element  $\vlam_P$ such that $\F_P\in\Bun_P^{\vlam_{P}}$.

\bigskip

If $\F_G$ is a principal $G$-bundle and $V$ a representation of $G$ then we will denote by $V_{\F_G} = \F_G\overset{G}{\times}V = (\F_G\times V)/G$ the associated vector bundle, where $\F_G\times V$ is endowed with the diagonal action of $G$. 

More generally, if $f:H\to G$ is a morphism of algebraic groups and $\F_H$ is an $H$-bundle then one can induce it to a $G$-bundle $\F_G:=\F_H\overset{H}{\times}G = (\F_H\times G)/H$ where $H$ acts diagonally. 
We will use this construction for the morphisms $P\to M, M\to P, P\to G$ where $P$ is a parabolic subgroup of $G$ and $M$ is its Levi quotient/subgroup.

\begin{defin}\label{D:reduction str gp} (reduction of structure group) Let $f:H\to G$ be a morphism of algebraic groups and let $\F_G$ be a $G$-bundle. A reduction (of structure group) of $\F_G$ to $H$ is a pair $(\F_H,\varphi)$ of an $H$-bundle together with an isomorphism of $G$-bundles $\varphi:\F_H\overset{H}{\times}G\to \F_G$.
\end{defin}

\subsection{The slope map} 

When beginning to learn about semistable $G$-bundles the first thing one remarks is that the definition of semistability doesn't resemble at all the classical one for vector bundles. 
The slope for vector bundles is particularly useful in defining and studying semistable vector bundles because it is intuitive and easy to define and work with. 
Ramanathan has given in~\cite{Ram-stable} several definitions of semistability which generalize naturally to any reductive group. Behrend has given a treatment of semistability and canonical reduction (Harder-Narasimhan reduction) in \cite{Beh_sstab} valid more generally for group schemes over curves. 

We will adopt here a slightly different point of view, due to Schieder~\cite{Schi}, which is closer in spirit to the vector bundles situation and also equivalent to the semistability as defined by Ramanathan. 
More precisely, Schieder defines a slope for $G$-bundles (where $G$ is an arbitrary reductive group) and then he mimicks the classical definition of semistability. 
He proves that the two notions are equivalent.  
We refer to \cite[Lemma 3.3]{Schi} for details. 
For the convenience of the reader we recall here how the slope is defined as well as some of its basic properties and we refer to loc. cit. for a full discussion.

\begin{defin}\label{D:slope map}
Let $P$ be a parabolic and denote by $M$ its Levi subgroup. The inclusion $\vLam_{Z(M)}^\Q\hookrightarrow \vLam_G^\Q$ followed by the projection onto $\vLam_{G,P}^\Q$ is an isomorphism. The slope map $\phi_P$ is defined as the composition
\[
 \phi_P:\vLam_{G,P}\to\vLam_{G,P}^\Q\simeq \vLam_{Z(M)}^\Q\hookrightarrow\vLam_G^\Q.
\]
\end{defin}
This slope map appears, although not very explicitly, in the paper \cite{AtiBott}, section 10. Their definition is only for a particular type of parabolic but clearly it works for any parabolic. 
They do not give particular attention to this map though and do not develop the combinatorial foundations as in \cite{Schi}. I thank Alexandru Chirvasitu for bringing this reference to my attention.

Let us give some examples on how this slope map works in order to familiarize the reader with the notion.

\begin{enumerate}
 \item Let $G = \GL_n$ and let $\V$ be a vector bundle of degree $d$. We write $\vlam_i$ for the coordinate cocharacters of $T$. The coroots are given by $\valpha_i = \vlam_i-\vlam_{i+1}$. 
 Then the degree of $\V$, viewed as an element in $\vLam_{G,G}=\oplus_i \Z\vlam_i/\<\valpha_i\mid 1\le i\le n\>$, is $\vlam_G = d\vlam_1 =\dots= d\vlam_n$. 
 The slope of $\V$ is \[
    \phi_G(\vlam_G) = \frac{d}{n}(\vlam_1+\dots+\vlam_n).
 \]
The appearence of $\frac1n$ is due to the inverse of the map $\vLam_{Z(M)}^\Q\to \vLam_{G,P}^\Q$. 
 
 \item The same as before, just that this time we look at a subbundle $\mathcal{W}\subset \V$ where $\mathcal{W}$ is of rank $m$ and degree $e$. 
 This data is equivalent to a reduction $\F_P$ of $\V$ to the maximal parabolic of $\GL_n$ corresponding to the simple root $\alpha_{m}$. 
 The degree of $\F_P$ is $\vlam_P = e\vlam_1+(d-e)\vlam_{m+1}\in\vLam_{G,P}$ and the slope is
 \[
  \phi_P(\vlam_P) = \frac em(\vlam_1+\dots+\vlam_m)+\frac{d-e}{n-m}(\vlam_{m+1}+\dots+\vlam_n).
 \]
 It is clear how one generalizes this example to a multi-step flag.
 \item More generaly, let $\F_G$ be a $G$-bundle of degree $\vlam_G$ and let $V$ be a highest weight representation of $G$ of weight $\lambda$. 
 Then we have the following equation (see \cite[Proposition 3.2]{Schi})
 \begin{equation}\label{E:slope induced vbdle}
  \mu(V_{\F_G}) = \<\phi_{G}(\vlam_G),\lambda\>,
 \end{equation}
 where $\mu$ denotes the usual slope (i.e. degree divided by rank) for vector bundles.
\end{enumerate}

\begin{defin}(Semistability, see \cite[Section 2.2.3]{Schi} ) Let $\F_G$ be a $G$-bundle of degree $\vlam_G\in\vLam_{G,G}$. 
Then we call $\F_G$ semistable if for any element $\vlam_P\in\vLam_{G,P}$ and for any reduction (see Definition~\ref{D:reduction str gp}) $\F_P$ of $\F_G$ to $P$ of degree $\vlam_P$ we have
\[
 \phi_P(\vlam_P)\le\phi_G(\vlam_G),
\]
where we recall that $\le$ is understood in the sense of the partial order on $\vLam_G^\Q$ induced by the positive cocharacters.

We say moreover that $\F_G$ is \emph{stable} if strict inequality holds in the above for proper parabolics.  
\end{defin}

Let us introduce some notations. 
We will denote by $\Bun_G^{\vlam_G}$ the stack of $G$\nobreakdash-bundles of degree $\vlam_G$ and similarly for other groups. 
We also let $\Bun_G^{\vlam_G,\ss}$ stand for the (open, dense) substack of semistable bundles. 
For a parabolic $P$ we denote by $\Bun_P^{\vlam_P,\ss}$ the preimage of semistable $G$\nobreakdash-bundles\footnote{Warning: in \cite{Schi} the superscript $\ss$ doesn't have the same meaning.}.

\subsection{Deeper reductions}\label{sec_deepred} When we deal with several reductions to parabolic subgroups it is always important to look at the relative position of the reductions. 
We will recall here the deeper reductions which were constructed in \cite[Section 4.2]{Schi} since we will use them frequently. 

Let $P_1,P_2$ be two parabolic subgroups of $G$ and let us denote by $\I_1$ and by $\I_2$ their associated vertices in the Dynkin diagram. 
The Weyl groups of their Levi's will be denoted by $W_1$ respectively $W_2$. 
We also put $W_{1,2}\subset W$ to be a system of representants of minimal length of $W_1\backslash W\slash W_2$.

An element of $\Bun_{P_1}\times_{\Bun_G}\Bun_{P_2}$ gives a natural map $X\times S \to BP_1\times_{BG} BP_2\simeq P_1\backslash G\slash P_2$ and hence for two reductions $\F_{P_1},\F_{P_2}$ of a $G$-bundle $\F_G$ on $X$ we obtain a map
\[
 X\to P_1\backslash G\slash P_2 = \bigsqcup_{w\in W_{1,2}} P_1\backslash P_1 w P_2\slash P_2.
\]

\begin{defin}
We say that two reductions $\F_P,\F'_P$ are in relative position $w$ if the above map factorizes through $P_1\backslash P_1 w P_2\slash P_2$. 
If this happens only generically on $X$ then we say that they are generically in relative position $w$.
\end{defin}

For a fixed $w\in W_{1,2}$ let us define the following sets of roots:
\begin{align}
 \I_1' = \{i\in \I_1\mid \exists j\in\I_2: w(\alpha_j)=\alpha_i\}\\
 \I_2'=\{i\in \I_2\mid \exists j\in\I_1: w^{-1}(\alpha_j)=\alpha_i\}
\end{align}

To $\I_1',\I_2'$ we associate the parabolics $Q_1\subseteq P_1$ and $Q_2\subseteq P_2$. We denote by $L_1,L_2$ their corresponding Levi subgroups. Remark that the conjugation by $w$ sends $L_1$ isomorphically onto $L_2$.

The following Proposition (see \cite[Corrolary 4.1]{Schi}) establishes the existence of deeper reductions.

\begin{prop}\label{prop_deepred} If $\F_{P_1}$ and $\F_{P_2}$ are two reductions of a $G$-bundle which are (generically) in relative position $w$ then there exist reductions $\F_{Q_1}$ of $\F_{P_1}$ and $\F_{Q_2}$ of $\F_{P_2}$ such that $\F_{Q_1}$ and $\F_{Q_2}$ are still (generically) in relative position $w$ and such that their induced Levi bundles $\F_{L_1}$ and $\F_{L_2}$ are naturally isomorphic when $L_1$ and $L_2$ are identified via the conjugation by $w$.
\end{prop}

\subsection{Properties of the slope map}

For the reader's convenience we collect under some lemmas a few of the  fundamental properties of the slope map from~\cite{Schi} that we will use in this note.

\begin{lemma}\label{L:slope_scalar}(\cite[Remark end of Section 2.1]{Schi}) The slope map has the following properties
\[ \< \vlam_P,\lam_P \> = \<\phi_p(\vlam_P),\lam_P\>, \, \forall \lam_P\in\Lam_{G,P}, \]
\[
\<\phi_P(\vlam_P),\alpha_i\> = 0, \,\forall i\in \I_M.
\]
\end{lemma}

\begin{lemma}\label{L:slope_order}(\cite[Proposition 3.1]{Schi})
The map $\phi_P:\vLam_{G,P}^\Q \to \vLam_G^{\Q}$ preserves the natural partial orderings.
\end{lemma}
 
\begin{lemma}\label{L:slope_proj}(\cite[Lemma 3.1]{Schi})
Let $P$ and $P'$ be two parabolics in $G$ and assume $P\subset P'$. Let $\vlam_{P'}\in \vLam_{G,P'}$ and consider the  projection $\proj_P:\vLam_G^\Q \to \vLam_{G,P}^\Q$. Then we have
\[
\phi_P(\proj_P(\phi_{P'}(\vlam_{P'}))) = \phi_{P'}(\vlam_{P'}).
\]
\end{lemma}

\begin{lemma}\label{L:slope_deep_red_w}(\cite[Proposition 4.6]{Schi})
In the setting of Proposition~\ref{prop_deepred} the following inequality holds
\[
w^{-1}\phi_{Q_1}(\vlam_{Q_1})\ge \phi_{Q_2}(\vlam_{Q_2}).
\]
\end{lemma}

\subsection{Drinfel'd's compactification of $\Bun_P$}\label{S:compactif_BunB} Let $P$ be a parabolic subgroup of $G$. In geometric Langlands, to define Eisenstein sheaves, one considers the induction map
\[
 p_P=p:\Bun_P\to\Bun_G
\]
which sends a $P$-bundle to the associated $G$-bundle (see end of Section~\ref{sec_prelim}). 
This map is not, in general, proper and so the objects constructed using it will not commute with Verdier duality. 
In order to fix this problem Braverman and Gaitsgory \cite{BG1} have studied a compactification of this morphism and showed that it posesses all the good properties one would like. 
We recall here their construction and some basic properties. 
For full details see \cite[Section 1.3]{BG1}. 
The author found the notes of T. Haines \cite{Hai} very useful.

Denote by $M$ the Levi of $P$.
The stack $\Bun_P$ classifies the data of
\[
 (\F_G,\F_{M/[M,M]},\kappa_P^\lambda,\lambda\in\Lam_{G,P}\cap\Lam_G^+)
\]
where $\F_G$ is a $G$-bundle, $\F_{M/[M,M]}$ is an $M/[M,M]$\nobreakdash-bundle and $\kappa_P^\lambda$ are maps of vector bundles
\[
 \kappa_P^\lambda:\L^\lambda_{\F_{M/[M,M]}}\hookrightarrow \V^\lambda_{\F_G}
\]
that satisfy the Pl\"ucker relations (see \cite[Section 1.2.1]{BG1}). 
We denoted by $\V^\lam_{\F_G}$ the vector bundle associated to the $G$-bundle $\F_G$ and the representation $V^\lam$ of $G$. Similarly for $\L^\lam_{\F_{M/[M,M]}}$. 

The compactification is obtained by relaxing the condition that $\kappa_P^\lam$ be a map of bundles, i.e. have no zeroes. We will only require it to be an injective map of coherent sheaves, i.e. the cokernel might have torsion.

The stack that we obtain in this way we denote by $\oBun_P$. It is an algebraic stack and it comes equipped with a proper map $\ov{p}_P:\oBun_P\to\Bun_G$ (see loc. cit. Section 1.3.2.).

\begin{rem} In \cite{BG1} the authors considered only the case where $G$ has simply connected derived group $[G,G]$. However, this restriction is only in order to have the property that $\Bun_P$ is dense in $\overline{\Bun}_P$ (see \cite[Proposition 1.2.3]{BG1}) and we do not use the density in this paper. 
We mention that a compactification with all the good properties of $\Bun_P\to \Bun_G$ was constructed in \cite[Section~7]{Schi} for an arbitrary reductive group $G$ but we will not make use of it. 
The only moment where we use Drinfeld's compactifification is to prove Proposition~\ref{proper_map} and for this it is enough to use the above described ``weak'' compactification.
\end{rem}

\subsection{Reductions of semistable $G$-bundles} 
\begin{defin}\label{D:admissible} A $P$-reduction $\F_P$ of a $G$-bundle $\F_G$ is called \emph{admissible} if
\[
\phi_P(\vlam_P) = \phi_G(\vlam_G),
\]
where $\vlam_G,\vlam_P$ are the degrees of $\F_G$ and $\F_P$.
\end{defin}

The main observation that started this work is the following simple lemma:

\begin{lemma}\label{the_parabolic}
 Let $\vlam_G$ be an element of $\vLam_{G,G}$. Then there exists a smallest parabolic $P=P_{\vlam_G}$ with the property that there exists $\vlam_P\in\vLam_{G,P}$ with the following two properties: 
 \begin{align}
  \phi_G(\vlam_G) &= \phi_P(\vlam_P)\label{E:parab-lemma-slope}\\
  \pi(\vlam_P)&=\vlam_G\label{E:parab-lemma-proj}
 \end{align}
where $\pi:\vLam_{G,P}\to\vLam_{G,G}$ is the natural projection. 
This parabolic subgroup is given by the following set of roots
 \[
  \I_{\vlam_G}:=\{i\in \I\mid \<\phi_G(\vlam_G)-\vlam,\omega_i\>\not\in\Z\},
 \]
 where $\vlam\in\vLam_G$ is a lift of $\vlam_G$.
 Moreover, the degree $\vlam_P$ is unique.
\end{lemma}
\begin{proof}
Remark that the definition of $I_{\vlam_G}$ above does not depend on the choice of the lift $\vlam$ since for two different such choices $\vlam$ and $\vlam'$ we have $\<\vlam-\vlam',\omega_i\>\in\Z$.

Uniqueness is straightfoward: suppose there are $\vlam_P,\vlam'_P\in\vLam_{G,P}$ that satisfy equations (\ref{E:parab-lemma-slope}),(\ref{E:parab-lemma-proj}). This implies $\vlam_P-\vlam'_P\in \ker(\pi)\cap\ker(\phi_P)$. But the group $\ker(\pi)=\<\valpha_j\mid j\in\I-\I_P\>$ is free abelian and the group $\ker(\phi_P)=\text{torsion}(\vLam_{G,P})$ is torsion, hence their intersection is $0$.

\bigskip

For the existence, we will first prove that for any parabolic subgroup $P$ for which there exists $\vlam_P\in\vLam_{G,P}$ that verifies the two equations (\ref{E:parab-lemma-slope}), (\ref{E:parab-lemma-proj}) above we have $\I_P\supseteq \I_{\vlam_G}$.

Since $\vlam_P\in\vLam_{G,P}$ it follows that $\<\vlam_P,\lam\>$ is well defined for any $\lam\in\Lam_{G,P}^\Q$. 
In particular, $\<\vlam_P,\om_i\>$ is defined for any $i\not\in \I_P$. 

From the definition (see also \cite[Equation 2.1]{Schi}) we have that $\<\phi_P(\vlam_P),\lam\> = \<\vlam_P,\lam\>$ for any $\lam\in\Lam_{G,P}^\Q$. 
In particular, we have 
\begin{align}\label{E:eq form lemma parabolic}\<\phi_P(\vlam_P),\om_i\>=\<\vlam_P,\om_i\> \quad\text{ for all } i\not\in \I_P.
\end{align}

Pick $\vlam$ a lift of $\vlam_P$ to $\vLam_G$. It is therefore also a lift of $\vlam_G$.

The equation (\ref{E:eq form lemma parabolic}) implies that 
\[\<\phi_P(\vlam_P)-\vlam,\om_i\>=0 \quad \text{ for all } i\not\in \I_P.
\]
Since $\phi_P(\vlam_P)=\phi_G(\vlam_G)$ we get that $\I_{\vlam_G}\subseteq \I_P$.

It remains to show that the parabolic $P=P_{I_{\vlam_G}}$ corresponding to $\I_{\vlam_G}$ works.
Pick $\vlam$ a lift of $\vlam_G$ to $\vLam_G$. Up to replacing it by 
\[
\vlam+\sum_{i\not\in \I_{\vlam_G}}\<\phi_G(\vlam_G)-\vlam,\om_i\>\valpha_i 
\]
we can suppose $\<\phi_G(\vlam_G)-\vlam,\omega_i\>=0,\,\forall i\not\in \I_{\vlam_G}$. Let us denote by $\vlam_P$ its projection to $\vLam_{G,P}$. We claim that $\vlam_P$ satisfies the requirements in the statement. Equation (\ref{E:parab-lemma-proj}) is clear by construction.

To check equation (\ref{E:parab-lemma-slope}) we observe that by the definition of the slope map we have 
$\phi_P(\vlam_P)) = \vlam+\sum_{i\in \I_{\vlam_G}} r_i \valpha_i$ where $r_i\in\Q$
 are uniquely determined by  $\<\vlam+\sum_{i\in \I_{\vlam_G}} r_i \valpha_i,\alpha_j\> = 0,\forall j\in \I_{\vlam_G}$.

Similarly, $\phi_G(\vlam_G) = \vlam+\sum_{i\in \I} s_i\valpha_i, s_i\in\Q$ where the $s_i$ are uniquely determined by
$\<\vlam+\sum_{i\in \I} s_i\valpha_i,\alpha_j\> = 0,\forall j\in \I$.

From $\<\phi_G(\vlam_G)-\vlam,\omega_j\> = 0,\forall j\not\in \I_{\vlam_G}$ we get that $s_j=0,\,\forall j\not\in \I_{\vlam_G}$. The uniqueness of the $(r_i)_i$ implies that $r_i=s_i,\forall i\in \I_{\vlam_G}$ and therefore $ \phi_P(\vlam_P)=\phi_G(\vlam_G)$. 
\end{proof}

\begin{rem}\label{R:sstable are stable}
If in the above lemma $P=G$ then no semistable $G$-bundle admits an admissible reduction to a parabolic subgroup and hence any semistable $G$-bundle of degree $\vlam_G$ is automatically stable.
\end{rem}

\textbf{Examples.} Let us see what this lemma gives in concrete exemples in the case of $\GL_n$.
Let $G=\GL_6$ and let us take $\vlam_G = d\vlam_6$ with $d$ an integer between $0$ and $5$. Then $\phi_G(\vlam_G) = \frac d6(\vlam_1+\dots+\vlam_6)$. Recall that $d$ corresponds, classically, to the degree of the vector bundle.
\begin{enumerate} 
 \item If $d = 1,5$ then $\I_{\vlam_G} = \I$ and the parabolic is $P=G$. 
 \item If $d = 2,4$ then $\I_{\vlam_G} = \{1,2,4,5\}$. 
 \item If $d = 3$ then $\I_{\vlam_G} = \{1,3,5\}$.
 \item If $d=0$ then $\I_{\vlam_G} = \emptyset$ and the parabolic is $P=B$.
\end{enumerate}

The following two lemmas do not play an essential role in the proof of the main result but we felt that they answer a natural question  regarding the interplay between semistability for $M,P$ respectively $G$-bundles, so we included them.

\begin{lemma}\label{lem_Msst_Gsst} Let $\vlam_G\in\vLam_{G,G}$ and $\F_P$ be a $P$-bundle of degree $\vlam_P$ as in Lemma~\ref{the_parabolic}. 
Let us denote by $M$ the Levi of $P$. 
If the induced $M$-bundle $\F_P/R_u(P)$ is semistable then the induced $G$\nobreakdash-bundle along the inclusion $P\hookrightarrow G$ is also semistable.
\end{lemma}
\begin{proof} Let us denote by $\F_G$ the $G$-bundle from the statement. 
Its degree is precisely $\vlam_G$.
Let $Q$ be a parabolic subgroup of $G$ and suppose there exists a reduction $\F_P$ of $\F_G$ to $Q$ which is of degree $\vlam_Q$.
We want to prove that $\phi_Q(\vlam_Q)\le \phi_G(\vlam_G)$.

Suppose the bundles $\F_P$ and $\F_Q$ are, generically, in relative position $w\in W$. We will use once again the construction of deeper reductions of \cite[Section 4.2]{Schi}, see Proposition~\ref{prop_deepred}. Let $P_1\subseteq P$ and $Q_1\subseteq Q$ be the parabolics given by this construction.
Since $\F_P$ is semistable when induced to an $M$-bundle we get from \cite[Lemma 4.7]{Schi} the following inequality\footnote{Some confusions might arrise here because the notations are not identical with those of the cited reference.}
\[
 w^{-1}\phi_P(\vlam_P)\ge w^{-1}\phi_{P_1}(\vlam_{P_1}).
\]
Moreover, Lemma~\ref{L:slope_deep_red_w} gives us
\[
 w^{-1}\phi_{P_1}(\vlam_{P_1})\ge \phi_{Q_1}(\vlam_{Q_1}).
\]

Combining these two inequalities with $\phi_P(\vlam_P) = \phi_G(\vlam_G)$ and the fact that the latter is invariant under $W$ we obtain 
\begin{equation}\label{eq_GQ1}
\phi_G(\vlam_G)\ge\phi_{Q_1}(\vlam_{Q_1}).  
\end{equation}

Let us denote by $\proj_Q:\vLam^\Q_G\to \vLam_{G,Q}^\Q$ the natural projection. It's clear that it preserves the partial orders.
We obviously have $\vlam_Q = \proj_Q(\vlam_{Q_1})$. 
Moreover, from Lemma~\ref{L:slope_proj} we have 
\begin{equation}\label{eq_QprojG}
\phi_Q(\proj_Q(\phi_G(\vlam_G))) = \phi_G(\vlam_G) 
\end{equation}

 and from the definition of the slope map we also have 
\begin{equation}\label{eq_QprojQ1}
\phi_Q(\proj_Q(\phi_{Q_1}(\vlam_{Q_1}))) = \phi_Q(\vlam_Q).
\end{equation}

Lemma~\ref{L:slope_order} says that $\phi_Q$ preserves the partial orders. 
Applying $\proj_Q$ and then $\phi_Q$ to the inequality (\ref{eq_GQ1}) and using the equalities~(\ref{eq_QprojG}),(\ref{eq_QprojQ1}) we obtain
\[
 \phi_G(\vlam_G)\ge \phi_Q(\vlam_Q).
\]
which is what we wanted to prove.
\end{proof}

\begin{lemma}\label{lem_Gsst_Msst} If $\F_P$ is a $P$-bundle of degree $\vlam_P$ as in Lemma~\ref{the_parabolic} such that the induced $G$-bundle is semistable then the induced $M$-bundle is also semistable.
\end{lemma}
\begin{proof}
 This is obvious since the parabolics of $M$ are in natural bijection with the parabolics of $G$ contained in $P$.
\end{proof}

\begin{rem}\label{R:diag_cartes}
Lemmas~\ref{lem_Msst_Gsst}, \ref{lem_Gsst_Msst} show that the following diagram makes sense and both squares are cartesian:
 
 {\centering$\xymatrix{
 & \Bun_P^{\vlam_P,\ss}\ar@{^(->}[d] \ar[rd]^{p_P}\ar[ld]_{q_P} &\\
 \Bun_M^{\vlam_P,\ss} \ar@{^(->}[d] & \Bun_P^{\vlam_P}\ar[rd]_{p_P}\ar[ld]^{q_P} & \Bun_G^{\vlam_G,\ss}\ar@{^(->}[d]\\
 \Bun_M^{\vlam_P} & & \Bun_G^{\vlam_G}
 }$
 
 }
We also obtain that the following induction map is well defined
 \begin{equation}\label{E:M_to_G_sst}
  \Bun_M^{\vlam_P,\ss}\longrightarrow\Bun_G^{\vlam_G,\ss}. 
 \end{equation}

\end{rem}

\begin{defin}
Let $P$ be a parabolic subgroup of $G$ and denote by $M$ its Levi subgroup. We define the regular $M$-bundles to be those $M$-bundles $\F_M$ for which the cohomology groups $H^\bullet(X,{(\g\slash \p)}_{\F_M})$ vanish.
\end{defin}

We will denote by $\Bun_M^{\vlam_P,\reg}$ the substack of regular bundles of degree $\vlam_P$.  Similarly, one defines regular $P$-bundles by the condition $H^\bullet(X,(\g/\p)_{\F_P}) = 0$. We will denote the corresponding substack by $\Bun_P^{\vlam_P,\reg}$.
Using the projection $P\to M$ one sees that $\Bun_P^{\vlam_P,\reg}$ is the preimage of $\Bun_M^{\vlam,\reg}$ under the natural map.

\begin{rem} It follows from general considerations of semicontinuity that the regular locus is open. We will see in the course of the proof of our main theorem that it is also non empty (for the particular degree $\vlam_P$ given by Lemma~\ref{the_parabolic}).
\end{rem}


\subsection{The Weyl group action is generically free}\label{S:Weyl gen free} In this subsection $X$ is an arbitrary smooth projective curve of nonzero genus.

\begin{lemma}\label{L:act-free} Let $P$ be a parabolic subgroup of $G$ and denote by $M$ its Levi subgroup. We denote by $W_{M,G}$ the relative Weyl group of $M$ in $G$. Let $\vlam\in\vLam_{G,P}$ be a degree such that $W_{M,G}(\vlam)=\vlam$. Then the group $W_{M,G}$ acts on $\Bun_M^{\vlam}$ generically free on objects.
\end{lemma}
\begin{proof}
Observe that it is enough to prove that $W_{M,G}$ acts on the stack $\Bun_S^{\vlam}$ generically free on objects, where $S:=M/[M,M]$. Indeed, the projection morphism $M\to S$ is $W_{M,G}$-equivariant and hence the induced determinant map $\det:\Bun_M^\vlam\to \Bun_S^\vlam$ is also $W_{M,G}$-equivariant.

Lemma~\ref{L:rel weyl Levi act free} implies that $W_{M,G}$ acts faithfully on $S$. Then we conclude using Lemma~\ref{L:fgrp torus act free}.
\end{proof}

\begin{lemma}\label{L:rel weyl Levi act free} Let $P$ be a parabolic subgroup of $G$ and denote by $M$ its Levi subgroup. We denote by $W_{M,G}=N_G(M)/M$ the relative Weyl group of $M$ in $G$. Then $W_{M,G}$ acts faithfully on the torus $M/[M,M]$.
\end{lemma}
\begin{proof}
Let $\vmu:\bG_m\to Z(M)^\circ$ be a regular one parameter subgroup. Regular here means that $C_G(\vmu(\bG_m)) = M$. The existence of such subgroups is classical. 
Let $w\in N_G(M)$. 
It is enough to prove that if the commutator $[w,\vmu(t)]\in [M,M]$ for all $t\in \bG_m$ then $w\in M$.

Clearly, for all $t\in\bG_m$ we have $[w,\vmu(t)]\in Z(M)^\circ$ and if it belongs also to $[M,M]$ then $[w,\vmu(t)] = 1$ for all $t\in\bG_m$ because $Z(M)^\circ\cap [M,M]$ is finite and $Z(M)^\circ$ is connected. In other words $w\in C_G(\vmu(\bG_m))=M$.
\end{proof}

\begin{lemma}\label{L:fgrp torus act free} Let $S$ be a torus and $\Gamma$ a finite group acting faithfully\footnote{No $\gamma\in\Gamma\setminus\{1\}$ acts as identity on $S$.} by algebraic group homomorphisms  on $S$. Let $\vlam$ be a cocharacter of $S$ that is invariant under $\Gamma$. Then $\Gamma$ acts on $\Bun_S^\vlam$ generically free on objects.
\end{lemma}
\begin{proof}
Let $\gamma\in \Gamma-\{1\}$ and let $\vmu$ be a cocharacter of $S$ such that $\gamma(\vmu)\neq \vmu$. This exists since $S$ is generated by its cocharacters and $\gamma$ is not the identity on $S$. Let $\alpha:S\to\bG_m$ be a character of $S$ such that $\<\vmu,\alpha\>-\<\gamma(\vmu),\alpha\> = r\neq 0$. This is also possible since the pairing between characters and cocharacters of a torus is perfect. 
Finally, since $X$ is not the projective line we can pick a line bundle $\cL\in\Pic^0(X)$ such that $\cL^r\neq \cO$. 

Define $\cF_S:=\vmu(\cL)\in\Bun_S^0$ the induced $S$-bundle using the cocharacter $\vmu:\bG_m\to S$. We claim that $\gamma(\cF_S)\not\simeq \cF_S$. Indeed, by applying $\alpha$ we get 
\[\alpha(\cF_S) \otimes\alpha(\gamma(\cF_S))^{-1} = \cL^{\<\alpha,\vmu\>-\<\alpha,\gamma(\vmu)\>} = \cL^r\not\simeq \cO.\]

The above proves that $\Gamma$ acts on $\Bun_S^0$ faithfully on objects. In fact, the action is generically free on objects. To see this, write $\Bun_S^0$ in the form $A\times BS$ where $A$ is a smooth algebraic variety (a product of the Jacobian $Pic^0(X)$ of $X$) and $BS$ is the classifying stack of $S$. The action of $W_{M,G}$ restricts to an action on $A$ and, since $A$ is separated, a non-trivial automorphism of $A$ acts generically freely. One can also remark that if we run the above construction for a generic $\vmu$ and a generic $\cL\in\Pic^0(X)$ one obtains an open dense substack of $\Bun_S^0$ on which $\gamma$ acts freely on objects. 

\vspace{0.2cm}
Let us prove that the same is true for the connected component $\Bun_S^\vlam$.

The group $S$ being commutative, the multiplication map $S\times S\to S$ is a group homomorphism and hence induces a tensor product $\Bun_S\times\Bun_S\to \Bun_S$ that we'll denote $(\cF_S, \cF'_S) \mapsto \cF_S\otimes \cF'_S$.

Since $\Gamma$ acts by group homomorphisms on $S$ we get that the above tensor product $\otimes:\Bun_S\times\Bun_S\to \Bun_S$ is $\Gamma$-equivariant.

Fix a point $x_0\in X$. Using the cocharacter $\vlam:\bG_m\to S$ we can induce the $\bG_m$-bundle $\cO(x_0)$ to an $S$-bundle that we denote $\cF_S(\vlam \cdot x_0)$. 
By the assumption on $\vlam$ we have $\gamma(\cF_S(\vlam\cdot x_0))\simeq\cF_S(\gamma(\vlam)\cdot x_0) = \cF_S(\vlam\cdot x_0)$ for all $\gamma\in \Gamma$. 

Since tensoring by $\cO(x_0)$ induces an isomorphism $\Pic^0(X)\to\Pic^1(X)$ we get that tensoring by $\cF_S(\vlam\cdot x_0)$ induces an isomorphism $\Bun_S^0\simeq \Bun_S^{\vlam}$. Moreover, from the previous paragraph, the isomorphism is also $\Gamma$-equivariant. We conclude that $\Gamma$ acts on $\Bun_S^\vlam$ generically free on objects.
\end{proof}

\section{Main theorem and proofs}

Let us first recall the definition of a small map.
\begin{defin} A map between algebraic varieties (or stacks) $f:Y\to Z$ is said to be \emph{semismall} if it is proper and
\[                                                                                              \dim(Y\times_Z Y) = \dim (Z).
\]
Moreover, if all irreducible components of $Y\times_Z Y$ which are of maximal dimension dominate $Z$ then we say that $f$ is {\emph{small}}.
\end{defin}

\begin{thm}\label{thm_main} Let $X$ be an elliptic curve. Let $\vlam_G$ be an element of $\vLam_{G,G}$ and let $P$ and $\vlam_P$ be those from Lemma~\ref{the_parabolic}.
Then the following holds:
 \[
  p:\Bun_P^{\vlam_P,\ss}\to\Bun_G^{\vlam_G,\ss}
 \]
 is a small map which is generically (over the regular locus) a Galois covering with Galois group the relative Weyl group $W_{M,G} = N_G(M)/M$.
\end{thm}
\begin{rem} It will follow from the proof of the main theorem that if $X$ is of  genus bigger than $1$ the map $p$ is small and \emph{birational} on its image and therefore it is a \emph{small resolution of singularities} of its image in $\Bun_G^{\vlam_G,\ss}$. 

For $G=\GL(n)$ this provides small resolutions of singularities of some particular Brill-Noether loci in the moduli stack  of semistable vector bundles.
\end{rem}

The proof of this theorem will be made through a series of lemmas which could be of independent interest.

\begin{prop}\label{proper_map} With the notations of Lemma~\ref{the_parabolic} we have
 that 
 \[
p:\Bun_P^{\vlam_P,\ss}\to\Bun_G^{\vlam_G,\ss}.
 \]
 is a proper map.
\end{prop}
\begin{proof}
We know from \cite[Section 1]{BG1} that the map $\overline{p}:\overline{\Bun_P}^{\vlam_P,\ss}\to\Bun_G^{\vlam_G,\ss}$ is proper so what we need to show is that in our situation a generalized reduction $(\F_G,\F_{M/[M,M]},k^\lambda)$ of $\F_G$ to $P$ (see Section~\ref{S:compactif_BunB}) is actually a true reduction to $P$. 
By our assumption $\phi_P(\vlam_P) = \phi_G(\vlam_G)$ and the saturation $\F'_P$ \cite[Section 1.3.3]{BG1} of the generalized reduction has degree $\vlam'_P$ that satisfies $\vlam'_P-\vlam_P\in\Q_+\{\valpha_i\mid i \in \I-\I_P\}$. 
Now using Lemma~\ref{L:slope_order} we get that $\phi_P(\vlam'_P)\ge \phi_P(\vlam_P) = \phi_G(\vlam_G)$ which contradicts the semistability of $\F_G$, unless $\vlam_P= \vlam'_P$, or in other words $(\F_G,\F_{M/[M,M]},k^\lambda)$ defines a reduction to $P$.
\end{proof}

\begin{lemma} Let $\F_P,\F'_P$ be two reductions of a semistable $G$-bundle $\F_G$ to a parabolic $P$ of degree $\vlam_P$ that are admissible (i.e. $\phi_P(\vlam_P) = \phi_G(\vlam_G)$). 
Suppose that the reductions are generically in relative position $w$. 
Then they are in relative position $w$ everywhere.
\end{lemma}
\begin{proof}
Recall the deeper reduction from Section~\ref{sec_deepred} and let us consider first the case when our reductions are equal to their deeper ones. This amounts to suppose that $w\in N_G(M)$.

From \cite[Proposition 4.5]{Schi} being generically in relative position $w$ means that the map
\[
V^\lambda[\lambda+\Z R_M]_{\F_M}\hookrightarrow V^\lambda_{\F_G}
\]
factorizes through 
\[
V^\lambda[\ge w\lambda+\Z R_M]_{\F'_P}\hookrightarrow V^\lambda_{\F_G}
\]
and induces an injective morphism \cite[Proposition 4.5]{Schi} of vector bundles 
\[
V^\lambda[\lambda+\Z R_M]_{\F_M} \longrightarrow V^\lambda[w\lambda+\Z R_M]_{\F'_M}                            
\]
which is an isomorphism on an open nonempty subset of $X$.
The first vector bundle has slope $\<\lambda,\phi_P(\vlam_P)\>$ (see equation (\ref{E:slope induced vbdle})) and the second has slope $\<w\lambda,\phi_P(\vlam_P)\> = \<\lambda,w^{-1}\phi_G(\vlam_G)\> = \<\lambda,\phi_G(\vlam_G)\>$ since the action of the Weyl group on $\vLam_{G,G}$ is trivial. 
An injective map of vector bundles of the same slope which is generically an isomorphism is an isomorphism. 
This proves that the two reductions are in relative position $w$ everywhere.

In the general situation we want to prove first that the slopes of the deeper reductions are the same as the slopes of the $P$-bundles. We have
\[
\phi_P(\vlam_P)-w^{-1}\phi_{Q_1}(\vlam_{Q_1}) = \sum_{i\in \I_P} n_iw^{-1}\valpha_i \ge 
\phi_P(\vlam_P)-\phi_{Q_2}(\vlam_{Q_2}) = \sum_{i\in \I_P} m_i\valpha_i\ge 0.
\]
Using \cite[Lemma 4.3]{Schi} we obtain that $n_i=0$ if $w^{-1}\valpha_i\not\in R_P$. Now this implies that $\phi_P(\vlam_P) = \phi_{Q_1}(\vlam_{Q_1}) = \phi_{Q_2}(\vlam_{Q_2})$ because of the non-degeneracy of the pairing between the characters and cocharacters of $Q_1$.

The same argument as in the first case proves that the deeper reductions are in relative position $w$ everywhere.
\end{proof}

\begin{cor}[of the proof]\label{deep_red_adm} Let $\F_P,\F'_P$ be two reductions of a $G$-bundle $\F_G$. Suppose that they are in relative position $w$ and that they are admissible (see Definition~\ref{D:admissible}). Then the deeper reductions (see Section~\ref{sec_deepred}) are also admissible.
\end{cor}

\begin{lemma}\label{rel_pos_normalizer} Using the notations of Lemma~\ref{the_parabolic}, if we have two reductions $\F_P,\F'_P$ of a semistable $G$-bundle, both of which are of degree $\vlam_P$, which are in relative position $w$, then $w\in N_G(M)$. 
\end{lemma}
\begin{proof}
 Being in relative position $w$ means that we have a section
 \[
  s:X\to (\F_P\stackrel{P}{\times}G)/P
 \]
which lands in
\[
 \F_P\stackrel{P}{\times}PwP/P.
\]
But the quotient stack (actually variety) $PwP/P$ is isomorphic, as a left $P$-space, to $P/P\cap wPw^{-1}$. So the above section can be rewritten as
\[
 s:X\to \F_P/P\cap wPw^{-1}
\]
which is equivalent to a reduction of $\F_G$ to $G'=P\cap wPw^{-1}$. Since $G'\cap M$ is a parabolic subgroup of $M$ there is a unique parabolic subgroup of $G$, say $Q$, included in $P$, such that $Q\cap M = G'\cap M$. 
From Corollary~\ref{deep_red_adm} we obtain an admissible reduction of $\F_G$ to $Q$ which implies $P=Q$ because $P$ was minimal with this property (see Lemma~\ref{the_parabolic}). 
But this forces $w$ to normalize the subgroup $M$.
\end{proof}

\begin{prop}\label{irred_cp} Let $X$ be an elliptic curve and let $\vlam_G,P$ and $\vlam_P$ be as in Lemma~\ref{the_parabolic}. 
The irreducible components of $\Bun_P^{\vlam_P,\ss}\times_{\Bun_G}\Bun_P^{\vlam_P,\ss}$ are in bijection with the relative Weyl group $W_{M,G}$ and are all of dimension $0$.
\end{prop}

\begin{proof} There is a natural map
\[
 \Bun_P^{\vlam_P,\ss}\times_{\Bun_G}\Bun_P^{\vlam_P,\ss} \to P\backslash G\slash P = \sqcup_{w\in W} P\backslash PwP\slash P
\]	
which sends a pair of reductions to their relative position at some fixed point, say $x_0$, of the curve. From Lemma~\ref{rel_pos_normalizer} we know that only the positions $w\in W_{M,G}$ can occur. Moreover, from the proof of the same lemma, the preimage of $P\backslash PwP\slash P$ is isomorphic to the stack $\Bun_{P\cap wPw^{-1}}^{\vlam_P,\ss}$. 
The latter is connected and, since $X$ is of genus 1, of dimension $0$. 

It follows that the irreducible components of the desired fibered product are in bijection with $W_{M,G}$ and are all of dimension $0$.
\end{proof}

\begin{lemma}\label{etale_map} Let $X$ be an elliptic curve.
With the notations of Lemma~\ref{the_parabolic} we have that the map 
\[
p:\Bun_P^{\vlam_P,\ss}\to\Bun_G^{\vlam_G,\ss}
\]
is \'etale (exactly) on the regular locus (which is not empty).
\end{lemma}
\begin{proof}
From the dimension estimates of Lemma~\ref{irred_cp} it follows that the map $p$ is generically quasi-finite and dominant.
Using generic smoothness we deduce that $p$ is \'etale on a non-empty open subspace of $\Bun_P^{\vlam_P,\ss}$. 

The tangent complex $\bT_p$ of $p$ sits in an exact triangle:
\[
 \bT_{\Bun_P^{\vlam_P,\ss}}\longrightarrow p^*\bT_{\Bun_G^{\vlam_G,\ss}}\longrightarrow {\bT}_{p}[1]\stackrel{+1}{\longrightarrow}
\]
and $p_P$ is \'etale exactly on the vanishing locus of $\bT_{p}$. 

By taking the fiber of the above exact triangle at some point $\F_P\in\Bun_P^{\vlam_P,\ss}$ we obtain the long exact sequence:

\begin{align*}
0\to H^0(X,\p_{\F_P}) & \stackrel{a'}{\to} H^0(X,\g_{\F_P})\to H^0(X,(\g/\p)_{\F_P})\to \\
& \to H^1(X,\p_{\F_P})\stackrel{a}{\to} H^1(X,\g_{\F_P})\to H^1(X,(\g/\p)_{\F_P})\to 0
\end{align*}

The map $p_P$ is \'etale at $\F_P$ if and only if $a$ and $a'$ are isomorphisms. 
This in turn is equivalent to $H^i(X,(\g\slash\p)_{\F_P}) = 0$ for $i=0,1$, i.e. $\F_P$ is regular.
\end{proof}

\begin{lemma}\label{L:M-reg iso P-reg} For $X$ of genus 1, keeping the notations of Lemma~\ref{the_parabolic}, we have that the restriction of 
\[
q:\Bun_P^{\vlam_P,\ss}\to \Bun_M^{\vlam_P,\ss}
\]
to the regular locus is an isomorphism.
\end{lemma}
\begin{proof} The same argument as in \cite[Proposition 3.2]{Lasz} works.
\end{proof}
\begin{rem}
Both the regularity and genus 1 are important for this isomorphism to hold. 
\end{rem}

\begin{proof}[Proof of Theorem~\ref{thm_main}] 

Let us first prove that the group $W_{M,G}$ acts on the connected component $\Bun_M^{\vlam_P}$. 
It amounts to verifying that  $W_{M,G}(\vlam_P)=\vlam_P$. This makes sense since $W_{M,G}$ normalizes $M$ and hence acts on $\vLam_{M,M}=\vLam_{G,P}$. 
Moreover, from the definition of the slope map one sees that $\phi_P$ is $W_{M,G}$-equivariant. 
The same is true of the projection $\pi:\vLam_{G,P}\to\vLam_{G,G}$. 
Note also that $W_{M,G}$ acts trivially on $\vLam_{G,G}$ and hence on the image of $\phi_G$ as well as on the image of $\pi$. 

Let $w\in W_{M,G}$. From the above we have $\pi(w\vlam_P) = w\pi(\vlam_P) = \vlam_G$ and $\phi_P(w\vlam_P) = w\phi_P(\vlam_P) = w\phi_G(\vlam_G) = \phi_G(\vlam_G)$.

We can apply the uniqueness from Lemma~\ref{the_parabolic} to $\vlam_P$ and $w\vlam_P$ and conclude that $w\vlam_P = \vlam_P$.

Therefore $W_{M,G}$ acts on the stack $\Bun_M^{\vlam_P}$ and from Lemma~\ref{L:act-free} we know that the action is generically free on objects.

By applying Lemma~\ref{L:M-reg iso P-reg}, we see that on an open subset of $\Bun_P^{\vlam_P,\ss,\reg}$ the group $W_{M,G}$ acts freely and the projection $p$ is invariant for this action (since $W_{M,G}$ acts by inner (!) automorphisms on $G$).

Furthermore, since $\Bun_P^{\vlam_P,\ss,\reg}$ dominates
$\Bun_G^{\vlam_G,\ss}$, the same is true for all the irreducible components of $\Bun_P^{\vlam_P,\ss}\times_{\Bun_G^{\vlam_G,\ss}}\Bun_P^{\vlam_P,\ss}$ and hence the smallness.

A proper \'etale map which is generically Galois is actually Galois hence the map $p$ is a Galois covering with Galois group $W_{M,G}$ when restricted to the regular locus.
\end{proof}

\section{Complements}

\subsection{Stable bundles}

In this section the curve $X$ will always be of genus 1.

\begin{prop}\label{prop_stable} If in Lemma~\ref{the_parabolic} $P=G$ then the adjoint group $G^\ad$ is isomorphic to a product $\prod_k \PGL_{n_k}$ and the degree $\vlam_G \cong (d_k)_k$ satisfies $\gcd(d_k,n_k) = 1,\,\forall k$. 
\end{prop}
\begin{proof}
Pick a lift $\vlam\in\vLam_G$ of $\vlam_G$. 
Let us translate in more concrete terms the meaning of the condition 
\begin{align}\label{E:for stable type A}\<\phi_G(\vlam_G)-\vlam,\om_i\>\not\in\Z,\forall i\in\I.
\end{align}

From the definition of the slope map $\phi_G$ the element $\phi_G(\vlam_G)$ has the following properties:
\begin{enumerate}
 \item $\phi_G(\vlam_G) = \vlam-\sum_{i\in\I}q_i \valpha_i$ with $q_i\in\Q,\,\forall i\in I$
 \item $\<\phi_G(\vlam_G),\alpha_i \rangle = 0,\,\forall i\in\I$.
\end{enumerate}
The equation (\ref{E:for stable type A}) implies that none of the $q_i$ above are integers.

If we denote by $C$ the Cartan matrix of the root system of $G$, by $q$ the vector $(q_i)_{i\in\I}$ and by $b$ the vector  $(\<\vlam,\alpha_i\>)_{i\in \I}$ (with integer entries(!) since $\vlam\in\vLam_G$) the second property above can be rewritten as
\[
 q = C^{-1}b.
\]

Since none of the entries of $q$ are integers and all entries of $b$ are integers we infer that on every line of $C^{-1}$ there is at least one non-integer number. 
Inspecting the inverses of Cartan matrices (see, for example, \cite[Table 2, page 295]{OnVin}) one sees immediately that this can happen only if the root system is a product of root systems of type $A$. 

Moreover, from the same tables, if $G^\ad\simeq \PGL_{k}$ in order for equation (\ref{E:for stable type A}) to hold, the degree $\vlam\in\vLam_{G}$ must also satisfy the condition
\[\gcd((\<\vlam,\alpha_i\>)_{i\in I},k) = 1. 
\]

To see why this must be so recall that the entries of the matrix $C^{-1}$, for type $A_{k-1}$ are $c_{i,j} = \frac1ki(k-j),i\le j$ and $c_{i,j} = \frac1k j(k-i),i>j$ where now the indices $i,j$ run in the set $\{1,\dots,k\}$. 
It is not hard to convince oneself that if the above greatest common divisor, say $e$, is at least $2$ then $q_{k/e}\in\Z$.

The above discussion implies that $\vlam_G$ projected to $\vLam_{G^\ad,G^\ad}$ has the form
\[
 d_1\vlam_{1,n_1-1}+d_2\vlam_{2,n_2-1}+\dots 
\]
where $G^\ad\simeq\prod_k \PGL_{n_k}$ and $\gcd(d_k,n_k) = 1,\forall k$ and we denoted by $\vlam_{k,i}$ the (standard) cocharacters for the projective linear group $\PGL_{n_k}$. 
\end{proof}

Here is an immediate combinatorial corollary:
\begin{cor}\label{C:Levi type A} In Lemma~\ref{the_parabolic} the Levi of $P$ must be of type products of type $A$.
\end{cor}

This corollary appears already in \cite{Bon}  as a consequence of the classification theorem \cite[Proposition 2.18]{Bon}. 

The next Corollary, although well-known over the complex numbers (see for example \cite[Proposition 2.9]{FrON2}), follows easily from the above:
\begin{cor}\label{cor_stable}
 There exists a stable $G$-bundle of degree $\vlam_G$ if and only if $G^\ad\simeq\prod_k \PGL_{n_k}$ and $(\vlam_G)^{\ad}\cong (d_k)_k$ satisfies $\gcd(d_k,n_k) = 1,\forall k$. 
\end{cor}
\begin{proof}
 Follows immediately from Corollary~\ref{C:Levi type A} and Theorem~\ref{thm_main} (see also Remark~\ref{R:sstable are stable}).
\end{proof}

\subsection{The Levi subgroups and the relative Weyl groups}\label{ss:levi-table}
In this subsection we provide a table with the Levi subgroups $M$, as well as their relative Weyl groups $W_{M,G}$, that come out from Lemma~\ref{the_parabolic} (cf. \cite[Table 2.17]{Bon}). 
For conciseness we do not write the Levi corresponding to degree 0 since they are always equal to the maximal torus.

\bigskip

\begin{center}
\begin{tabular}[b]{|c|c|c|c|c|}
\hline
$G$ & deg & \begin{tabular}{c}Type\\of M\end{tabular} & \begin{tabular}{c}Diagram \\ of $(G,M)$\end{tabular} & \begin{tabular}{c} Type of\\ $W_{M,G}$ \end{tabular}\\
\hline
\hline
$A_{n-1}$ & d & \begin{tabular}{c}$A_{n/e-1}\times\dots\times A_{n/e-1}$\\ $e=\gcd(n,d)$\end{tabular} & $\boxed{A_{n/e-1}}\!-\circ-\boxed{A_{n/e-1}}\!-\!\circ\dots-\circ-\!\boxed{A_{n/e-1}}$ & $A_{e-1}$\\
\hline
\hline
$B_n$ & 1 & ${A_1}^{\vphantom{\sum}}_{\vphantom{\sum}}$ & $\circ$---$\circ$---$\circ$---$\circ\dots\circ$---$\circ\!\Longrightarrow\!\!\bullet$ & $C_{n-1}$\\
\hline
\hline
$C_{2n}$ & 1 & $\underbrace{A_1\times A_1\cdots \times A_1}_{n}^{\vphantom{\sum}}$ & $\bullet$---$\circ$---$\bullet$---$\circ\cdots\cdots\bullet\!\!\Longleftarrow\!\!\circ$ & $C_n$\\
\hline
$C_{2n+1}$ & 1 & $\underbrace{A_1\times A_1\dots \times A_1}_{n+1}^{\vphantom{\sum}}$ & $\bullet$---$\circ$---$\bullet$---$\circ \cdots \cdots\circ\!\!\Longleftarrow\!\!\bullet$ & $C_n$\\
\hline
\hline
\multirow{4}{*}{$D_{2n+1}$} & 1 &$A_1\times \dots\times A_1\times A_3$  & $\xy
\POS(0,0) *{\bullet}="a",
\POS(10,0) *\cir<2pt>{}="b",
\POS(20,0) *{\bullet}="c",
\POS(30,0) *\cir<2pt>{}="x"
\POS(40,0) *{\bullet}="d",
\POS(50,5) *{\bullet}="e",
\POS(50,-5) *{\bullet}="f",

\POS(50,-7) *{}="tt"
\POS(50,7) *{}= "ss"

\POS "a" \ar@{-}^<<<<{} "b",
\POS "b" \ar@{-}^<<<<{} "c",
\POS "c" \ar@{.}^<<<<{} "x",
\POS "x" \ar@{-}^<<<<{} "d",
\POS "d" \ar@{-}^<<<<{} "e",
\POS "d" \ar@{-}^<<<<{} "f",
\endxy$ 
 & $C_{n-1}$ \\\cline{2-5}

& 2 &$A_1\times A_1$  & $\xy
\POS(0,0) *\cir<2pt>{}="a",
\POS(10,0) *\cir<2pt>{}="b",
\POS(20,0) *\cir<2pt>{}="c",
\POS(30,0) *\cir<2pt>{}="x"
\POS(40,0) *\cir<2pt>{}="d",
\POS(50,5) *{\bullet}="e",
\POS(50,-5) *{\bullet}="f",

\POS(50,-7) *{}="tt"
\POS(50,7) *{}= "ss"

\POS "a" \ar@{-}^<<<<{} "b",
\POS "b" \ar@{-}^<<<<{} "c",
\POS "c" \ar@{.}^<<<<{} "x",
\POS "x" \ar@{-}^<<<<{} "d",
\POS "d" \ar@{-}^<<<<{} "e",
\POS "d" \ar@{-}^<<<<{} "f",
\endxy$ 
 & $C_{n-1}$\\

\hline
\multirow{6}{*}{$D_{2n}$} & (1,0) &$A_1\times\dots\times A_1$ & $\xy
\POS(0,0) *{\bullet}="a",
\POS(10,0) *\cir<2pt>{}="b",
\POS(20,0) *{\bullet}="c",
\POS(30,0) *{\bullet}="x"
\POS(40,0) *\cir<2pt>{}="d",
\POS(50,5) *{\bullet}="e",
\POS(50,-5) *\cir<2pt>{}="f",
\POS(50,-7) *{}="tt"
\POS(50,7) *{}= "ss"

\POS "a" \ar@{-}^<<<<{} "b",
\POS "b" \ar@{-}^<<<<{} "c",
\POS "c" \ar@{.}^<<<<{} "x",
\POS "x" \ar@{-}^<<<<{} "d",
\POS "d" \ar@{-}^<<<<{} "e",
\POS "d" \ar@{-}^<<<<{} "f",
\endxy$ 
 & \begin{tabular}{c}$B_n$
\end{tabular}	\\\cline{2-5}
&(0,1) & $A_1\times A_1$ & $\xy
\POS(0,0) *\cir<2pt>{}="a",
\POS(10,0) *\cir<2pt>{}="b",
\POS(20,0) *\cir<2pt>{}="c",
\POS(30,0) *\cir<2pt>{}="x"
\POS(40,0) *\cir<2pt>{}="d",
\POS(50,5) *{\bullet}="e",
\POS(50,-5) *{\bullet}="f",
\POS(50,-7) *{}="tt"
\POS(50,7) *{}= "ss"

\POS "a" \ar@{-}^<<<<{} "b",
\POS "b" \ar@{-}^<<<<{} "c",
\POS "c" \ar@{.}^<<<<{} "x",
\POS "x" \ar@{-}^<<<<{} "d",
\POS "d" \ar@{-}^<<<<{} "e",
\POS "d" \ar@{-}^<<<<{} "f",
\endxy$ 
 & \begin{tabular}{c}$C_{2n-2}$
\end{tabular}	\\\cline{2-5}
 & (1,1) & $A_1\times\dots\times A_1$ & $\xy
\POS(0,0) *{\bullet}="a",
\POS(10,0) *\cir<2pt>{}="b",
\POS(20,0) *{\bullet}="c",
\POS(30,0) *{\bullet}="x"
\POS(40,0) *\cir<2pt>{}="d",
\POS(50,5) *\cir<2pt>{}="e",
\POS(50,-5) *{\bullet}="f",
\POS(50,-7) *{}="tt"
\POS(50,7) *{}= "ss"

\POS "a" \ar@{-}^<<<<{} "b",
\POS "b" \ar@{-}^<<<<{} "c",
\POS "c" \ar@{.}^<<<<{} "x",
\POS "x" \ar@{-}^<<<<{} "d",
\POS "d" \ar@{-}^<<<<{} "e",
\POS "d" \ar@{-}^<<<<{} "f",
\endxy$
 & $C_n$	\\
\hline
\hline
\multirow{3}{*}{$E_6$} & \multirow{3}{*}{1} & \multirow{3}{*}{$A_2\times A_2$} & 
$\xy 
\POS (-10,0) ="z",
\POS (0,0) *{\bullet} ="a" ,
\POS (10,0) *{\bullet} ="b" ,
\POS (20,-10) *\cir<2pt>{} ="c" ,
\POS(20,-13) *{}="tt"
\POS (20,0) *\cir<2pt>{} ="d" ,
\POS (30,0) *{\bullet} ="e" ,
\POS (40,0) *{\bullet} ="f" ,
\POS (50,0)  ="h", 

\POS "a" \ar@{-}^<<<{}_<<{}  "b",
\POS "b" \ar@{-}^<<<{}_<<{}  "d",
\POS "c" \ar@{-}^<<{}_<<{}  "d",
\POS "d" \ar@{-}^<<<{}_<<{}  "e",
\POS "e" \ar@{-}^<<<{}_<<{}  "f",
\POS "f" \ar@{}^<<<{}_<<{}  "h",
\endxy$
 & \multirow{3}{*}{$G_2$} \\
\hline
\hline
\multirow{3}{*}{$E_7$} & \multirow{3}{*}{1} & \multirow{3}{*}{$A_1\times A_1\times A_1$} & 
$\xy 
\POS (-10,0) ="z",
\POS (0,0) *\cir<2pt>{} ="a" ,
\POS (10,0) *\cir<2pt>{} ="b" ,
\POS (20,-10) *{\bullet} ="c" ,
\POS(20,-13) *{}="tt"
\POS (20,0) *\cir<2pt>{} ="d" ,
\POS (30,0) *{\bullet} ="e" ,
\POS (40,0) *\cir<2pt>{} ="f" ,
\POS (50,0)  *{\bullet}="h"

\POS "a" \ar@{-}^<<<{}_<<{}  "b",
\POS "b" \ar@{-}^<<<{}_<<{}  "d",
\POS "c" \ar@{-}^<<{}_<<{}  "d",
\POS "d" \ar@{-}^<<<{}_<<{}  "e",
\POS "e" \ar@{-}^<<<{}_<<{}  "f",
\POS "f" \ar@{-}^<<<{}_<<{}  "h"
\endxy$

&\multirow{3}{*}{$F_4$}\\
\hline
\end{tabular}
\end{center}

\bigskip

\section*{Acknoledgements} This work is part of my PhD thesis. I would like to express my gratitude to my advisor, Olivier Schiffmann, for suggesting this project and for his constant help and support. I would like to thank the Fondation des Sciences Math\'ematiques de Paris for the grant offered to visit Harvard University in the Fall of 2013 where much of this work was done. I'm also very grateful to Dennis Gaitsgory who agreed for me to visit him at Harvard. This work benefited a lot from discussions with him, Sam Raskin and Simon Schieder whom I warmly thank. I would like to thank the referee for pointing out a gap in the proof that lead to Section~\ref{S:Weyl gen free}. 
I warmly thank Michael Gr\"ochenig for reading an early draft and for his comments. 


\bibliographystyle{acm} 

\def\cprime{$'$}

\end{document}